\documentclass[a4paper,11pt,reqno]{amsart}
\usepackage{amsmath,amsthm,amssymb}
\usepackage{color}
\usepackage[body={16cm,22.8cm},centering]{geometry} 
\usepackage{enumitem} 
\usepackage{latexsym}
\usepackage{stmaryrd}
\usepackage[normalem]{ulem}   	

\theoremstyle{plain}
\newtheorem{theorem}{Theorem}[section]
\newtheorem{lemma}[theorem]{Lemma}
\newtheorem{proposition}[theorem]{Proposition}
\newtheorem{corollary}[theorem]{Corollary}
\newtheorem{definition}[theorem]{Definition}

\theoremstyle{remark}

\numberwithin{equation}{section}

\newcommand{\R}{\mathbb{R}} 
\newcommand{\N}{\mathbb{N}}

\newcommand{\Om}{\Omega}
\newcommand{\Sf}{\mathbb{S}}

\newcommand{\A}{\mathcal{A}}
\newcommand{\B}{\mathcal{B}}
\newcommand{\D}{\mathcal{D}}
\newcommand{\E}{\mathcal{E}}

\newcommand{\e}{\varepsilon}
\newcommand{\mc}{\mathcal}

\newcommand{\dd}{\mathrm{d}}
\renewcommand{\div}{\operatorname{div}}

\DeclareMathOperator{\cof}{cof}

\DeclareMathOperator{\imG}{im_G}
\DeclareMathOperator{\imT}{im_T}
\DeclareMathOperator{\loc}{loc}
\DeclareMathOperator{\adj}{adj}

\DeclareMathOperator{\rel}{rel}

\DeclareMathOperator{\sgn}{sgn}

\newcommand{\weakcs}{\overset{*}{\rightharpoonup}}
\renewcommand{\vec}[1]{\text{\boldmath $#1$}}

\mathchardef\emptyset="001F

\begin{document}
\title[]{A relaxation approach to the minimisation of the neo-Hookean energy in 3D}
\author{Marco Barchiesi}
\author{Duvan Henao}
\author{Carlos Mora-Corral}
\author{R\'emy Rodiac}
\date{\today}

\address[Marco Barchiesi]{
Dipartimento di Matematica, Informatica e Geoscienze, Universit\`a degli Studi di Trieste,
Via Weiss 2 - 34128 Trieste, Italy.
}
\email{barchies@gmail.com}

\address[Duvan Henao]{Instituto de Ciencias de la Ingenier\'ia, Universidad de O'Higgins. Rancagua, Chile \& Center for Mathematical Modeling.}
\email{duvan.henao@uoh.cl}

\address[Carlos Mora-Corral]{Departamento de Matem\'aticas, Universidad Aut\'onoma de Madrid,
28049 Madrid, Spain  and Instituto de Ciencias Matem\'aticas,
CSIC-UAM-UC3M-UCM, 28049 Madrid, Spain.
}
\email{carlos.mora@uam.es}

\address[R\'emy Rodiac]{Universit\'e Paris-Saclay, CNRS,  Laboratoire de math\'ematiques d'Orsay, 91405, Orsay, France \& Institute of Mathematics, University of Warsaw, Banacha 2, 02-097
Warszawa, Poland}
\email{rrodiac@mimuw.edu.pl}

\begin{abstract}
Despite its high significance in nonlinear elasticity, the neo-Hookean energy is still not known to admit minimisers in some appropriate admissible class. Using ideas from relaxation theory, we propose a larger minimisation space and a modified functional that coincides with the neo-Hookean energy on the original space.
This modified energy is the sum of the neo-Hookean energy and a term penalising the singularities of the inverse deformation.
The new functional attains its minimum in the larger space, so the initial question of existence of minimisers of the neo-Hookean energy is thus transformed into a question of regularity of minimisers of this new energy.
\end{abstract}

\maketitle


\section{Introduction}

\subsection{Overview of the problem}
The neo-Hookean model, given its widespread use, is highly significant in nonlinear elasticity. In this model, minimisers of the neo-Hookean energy
\begin{equation}\label{eq:neo_Hook_energy}
E(\vec u)=\int_\Om \left[ |D \vec u|^2+ H(\det D \vec u) \right] \dd\vec x
\end{equation}
are sought in a space of orientation-preserving maps (i.e., with \(\det D \vec u>0\) a.e.) satisfying some injectivity conditions (e.g., \(\vec u\) one-to-one a.e.) in order to avoid interpenetration of matter. Here \(H: (0, \infty) \rightarrow [0, \infty) \) is a convex function such that 
\begin{align}
	\label{eq:explosiveH}
\lim_{t\rightarrow \infty} \frac{H(t)}{t}=\lim_{s \rightarrow 0} H(s)=\infty,
\end{align}
\( \Om \subset \R^3 \) represents the reference configuration of an elastic body and  \( \vec u: \Om \rightarrow \R^3 \) is the deformation map.
Unfortunately, the coercivity of the neo-Hookean energy is not sufficient to apply the current theories in calculus of variations to deduce existence of a minimiser in an appropriate space. Indeed, the neo-Hookean energy  is a borderline case of energies that admit minimisers, like 
\[
 \int_\Om \left[ |D\vec u|^p+H(\det D \vec u) \right] \dd \vec x
\]
with \(p>2\) or
\[
 \int_\Om \left[ |D \vec u|^2+H(\det D \vec u)+ \widetilde{H}(|\cof D \vec u|)\right] \dd \vec x ,
\]
where \(\widetilde{H}\) is superlinear at infinity; cf., e.g., \cite{Ball77,BallMurat1984,Muller90,MuQiYa94,MuSp95,HeMo10,HeMo11,HeMo12,HeMo15} and references therein. 
The difficulty one has to face in minimising the neo-Hookean energy is due to the lack of compactness of the minimisation space with respect to the \(H^1\) convergence, as shown by an example of Conti \& De Lellis \cite{CoDeLe03}; see also \cite{BaHEMoRo_24PUB,Dolezalova_Hencl_Maly_2023}. For some results on existence of minimisers of the energy \eqref{eq:neo_Hook_energy} in the axisymmetric setting we refer to \cite{HeRo18,BaHEMoRo_23PUB}, but we emphasise that the goal of this article is to consider the general 3D case.

When one cannot prove the existence of minimisers via the direct method of calculus of variations, a common strategy consists in splitting the difficulty into two steps. The first step, called relaxation, aims at obtaining existence of minimisers of a modified energy in a bigger 
and less regular space, with the requirement that the modified energy coincides with the original one in the original space. The purpose of the second step is to prove regularity of one of the minimisers obtained in the previous step and to show that it actually belongs to the original smaller space. Our goal in this paper is to implement the first step for the minimisation of the neo-Hookean energy and to transform the existence problem into a regularity problem for a modified energy. The new energy we propose is motivated by our previous work \cite{BaHEMoRo_23PUB} where the same problem was considered in the particular case of axisymmetric deformations.

Before entering into details, we note that the process of relaxation gives rise to natural spaces in calculus of variations.
For instance,
minimising sequences of \( \|D u\|_{L^1 (\Om)} \) among \(W^{1,1}\) functions with prescribed Dirichlet data are not compact, and a larger space more suitable to the problem is the space of functions of bounded variation ($BV$).
Another example is the minimisation of the Dirichlet energy
on the space $H^1_{\vec b}(\Omega; \Sf^2)\cap C^0(\overline \Omega, \Sf^2)$ of continuous unit-valued $H^1$ maps with prescribed Dirichlet data $\vec b$ on $\partial \Omega$, $\Omega\subset \R^3$,
a problem extensively studied beginning with the pioneering works  \cite{BrCoLi86,Hardt_Lin_1986,BeBrCo90}.
Since $H^1_{\vec b}\cap C^0$ is not weakly compact, the relaxation leads to the minimisation of a modified energy functional in the larger space of unit-valued maps in $H^1_{\vec b}$ that satisfy the boundary condition but are not necessarily continuous.

\subsection{Setting and statement of the main result}

We now describe more precisely our minimisation setting. We work with a strong form of the Dirichlet boundary 
condition, namely, we choose  a smooth bounded domain \(\widetilde{\Om}\) of \(\R^3\) compactly included in $\Om$,
and we require that deformations coincide with a bi-Lipschitz orientation-preserving homeomorphism
 $\vec b:\Om \rightarrow \R^3$ not only on $\partial \Om$ but on the whole of
$\Om \setminus \widetilde{\Om}$. We define
\begin{equation*}
\Om_\vec b:= \vec b(\Om) \quad \text{ and } \quad \widetilde{\Om}_{\vec b} := \vec b(\widetilde{\Om}).
\end{equation*}
We require the deformations to be orientation preserving, i.e., to satisfy \(\det D \vec u>0\) a.e..
Since interpenetration of matter is physically unrealistic, we also ask that the deformations satisfy some injectivity conditions:
first of all to be one-to-one a.e. 
We recall that \(\vec u:\Om \rightarrow \R^3\) being one-to-one a.e.\ means that there exists a set \(N\) of zero Lebesgue measure such that \( \vec u|_{\Om \setminus N}\) is one-to-one. 
Secondly, since the interpenetration of matter is not merely injectivity, we would like to request the deformations to 
satisfy the well-known INV condition (see \cite{MuSp95, CoDeLe03}).
Simplifying, the INV condition means that after the deformation, matter coming from any subregion $U$ remains 
enclosed by the image of $\partial U$ and matter coming from outside $U$ remains exterior to the region enclosed 
by the image of $\partial U$.
Because of that, we will impose that maps in $\A$ (the admissible class) satisfy the \emph{divergence identities:}
\begin{equation}
\label{eq:divergence_identities}
\operatorname{Div}\, \big ( (\adj D\vec u)\vec g\circ \vec u \big ) = (\div \vec g)\circ \vec u \,\det D\vec u\quad \forall\,\vec g \in C_c^1(\R^3,\R^3) ,
\end{equation}
where $\operatorname{Div}$ is the distributional divergence in $\Om$.
Indeed, one can use the Brezis-Nirenberg degree and adapt \cite[Lemma 5.1]{BaHeMo17} to show that 
for maps satisfying the divergence identities condition INV holds.  
All these requirements lead us to try to work with the minimisation space
\begin{align*}
 \A := \{ \vec u \in H^1(\Om,\R^3) : \, & \vec u = \vec b \text{ in } \Om  \setminus \widetilde{\Om}, 
 \, \vec u \text{ is one-to-one a.e.}, \, \det D\vec u>0 \text{ a.e.},  \\ 
& \text{identity \eqref{eq:divergence_identities} holds and }  E(\vec u)<\infty  \}.
\end{align*}

Unfortunately, this space is not closed with respect to the \(H^1\) weak convergence, and one has to face a problem of lack of compactness, as shown by Conti \& De Lellis in their example \cite[Theorem 6.1] {CoDeLe03}. This non-compactness impedes the application of the direct method of calculus of variations.
As mentioned in the introduction, our strategy is the following: 
we seek a larger space $\B$  that is \emph{compact} for sequences with equibounded energy. 
On that space, we want a \emph{lower semicontinuous} energy $F$ coinciding with $E$ on \(\A\). 
By using the direct method of calculus of variations, one can then obtain that the energy $F$ admits a 
minimiser $\vec u$ on $\A$. Then, the existence problem  of a minimiser for $E$ is reduced to showing 
that $\vec u$ belongs (hopefully) to $\A$.

A natural candidate for the space $\B$ would be 
\begin{equation*}
\overline{\A}:=\{ \vec u \in H^1(\Om,\R^3): \exists \ (\vec u_n)_n \subset \A \text{ with } \sup_n E(\vec u_n)<\infty \text{ and } \vec u_n \rightharpoonup \vec u \text{ in }  H^1\} .
\end{equation*}
However, we prefer to work on an explicit space. Because of that we introduce a larger space $\B$ containing $\overline{\A}$. 
Our choice of the family $\B$ and the energy $F$ is driven by the following two facts
(Lemma \ref{le:two_facts}). First, the geometric image of a map \(\vec u\) in $\A$, defined in 
Definition \ref{def:geometric_image} and which can be thought of as \(\vec u(\Om)\), can be shown 
to be equal to ${\Omega}_{\vec b}$ in a measure theoretic sense (i.e., none of the space enclosed 
by $\vec u(\partial \Omega)$ is left void and no part of the body is mapped to the region exterior to that surface).
Second, the inverse of a map in $\A$ belongs to $W^{1,1}(\Om_{\vec b},\R^3)$.
However, this last condition is not stable: as shown by Conti \& De Lellis in their example, the weak $H^1$ limit 
of a sequence in $\A$ can have a limit with an inverse that is not in $W^{1,1}$ but only in $BV$.
This motivates us to define
\begin{align*}
\B:= \{\vec u \in H^1(\Om,\R^3) : \, & \vec u = \vec b \text{ in } \Om  \setminus \widetilde{\Om}, 
\, \vec u \text{ is one-to-one a.e.}, \, \det D \vec u\neq 0 \text{ a.e.}, \\
& \Om_{\vec b} = \imG(\vec u,\Om) \text{ a.e.}, \,
\vec u^{-1}\in BV(\Om_{\vec b},\R^3), \, \text{ and }  E(\vec u)<\infty  \},
\end{align*}
where $\imG(\vec u, \Om)$ is the geometric image defined in Definition \ref{def:geometric_image}. As explained before, the inclusion $\A\subset\B$ holds.
Then, we extend $E$ on $\B$ by defining
\begin{equation}\label{eq:F}
F(\vec u):= E(\vec u)+2 \| D^s \vec u^{-1} \|,
\end{equation}
for \(\vec u \in \B\), where, with a slight abuse of notation, \(E(\vec u)=\int_{\Om}|D \vec u|^2+H(|\det D\vec u|)\) for \(\vec u\) in \(\B\). We observe that in the space \(\B\) only the condition \(\det D\vec u\neq 0\) a.e.\ is required. This is because we were unable to show that the condition \(\det D\vec u>0\) a.e.\ passes to the limit for the weak \(H^1\) convergence of sequences in \(\A\) with equibounded \(F\) energy. We refer to \cite{Hencl_Onninen_2018} for such a result in the framework of Sobolev homeomorphisms.
Here $D^s \vec u^{-1}$ is the singular part of the distributional gradient of the inverse, 
$|D^s \vec u^{-1}|$ is the total variation of $D^s \vec u^{-1}$ (which is itself a positive Radon measure), and $\|D^s \vec u^{-1}\|$ is the norm of the measure $D^s \vec u^{-1}$, so that 
$\|D^s \vec u^{-1}\| = |D^s \vec u^{-1}| (\Om_{\vec b})$.

The definition of \(F\) is inspired by our previous works \cite{BaHEMoRo_23PUB} where we have proved that \(F\) admits a minimiser among axially symmetric maps belonging to \(\B\cap \{\vec u : \det D\vec u>0 \text{ a.e.} \}\).
In the present paper we extend this result to maps without any symmetry, but without the conclusion that in the limit \(\det D\vec u>0\) a.e.
Another feature of the energy \(F\) is that \(\|D^s \vec u^{-1}\|\) has an expression resembling the notion of minimal connections that was introduced by Brezis-Coron-Lieb in \cite{BrCoLi86} and which also appears in the relaxed energy for harmonic maps; cf.\ \cite{BeBrCo90}. We refer to \cite[Theorem 1.3]{BaHEMoRo_24PUB} for more on this expression. 

Our main theorem is the compactness of the class $\B$ and the lower semicontinuity of the functional $F$ 
with respect to the weak convergence in $H^1$ for maps in $\B$. 
This provides the existence of minimisers for the energy $F$ in $\B$ via the direct method of calculus of variations.

\begin{theorem}\label{main theorem}
Let $\{\vec u_j\}_j$ be a sequence in $\B$ such that $\{F(\vec u_j)\}_j$ is equibounded.
Then there exists $\vec u\in\B$ such that, up to a subsequence,
$\vec u_j \rightharpoonup \vec u$ in $H^1(\Om,\R^3)$ and 
\begin{equation*}
\liminf_{j\to\infty} F(\vec u_j)\geq F(\vec u).
\end{equation*} 
In particular, the energy $F$ has a minimiser in \(\B\). Moreover, $\overline{\A}\subset \B$.
\end{theorem}

\smallskip

We remark that, by definition of the relaxed energy, we have \( F(\vec u)\leq E_{\rel}(\vec u)\) for every \(\vec u\) in $\overline{\A}$. 
Here the relaxed energy is defined abstractly by  
\begin{equation*}
 E_{\rel}(\vec u):= \inf \{ \liminf_{j\rightarrow \infty} E(\vec u_j) : \{ \vec u_j \}_{j} \subset \A \text{ and } \vec u_j \rightharpoonup \vec u \text{ in } H^1 (\Om, \R^3) \}.
\end{equation*}
It is desirable that $F$ coincides with the relaxation of $E$, 
in order to get, possibly, a negative result: if none of the minimisers of the relaxed energy belong to $\A$, then $E$ has no minimisers in $\A$.
It is important to mention that the factor $2$ in formula \eqref{eq:F} appearing in front of $\| D^s \vec u^{-1} \|$ is sharp, as shown in \cite{BaHEMoRo_24PUB}:
there exists a map $\vec u$ in $\B\setminus\A$ (the nasty one provided by Conti \& De Lellis) and a sequence $\{\vec u_j\}_j$ in $\A$ 
such that  $\lim_{j\rightarrow \infty} E(\vec u_j) = F(\vec u)$. However, we are not able to prove that \(F\) coincides with the relaxed energy \(E_{\rel}\) at the moment.

A final remark is that we focus mainly on the Dirichlet part of the neo-Hookean energy, i.e., on $|D \vec u|^2$.
But some recent results in \cite{Dolezalova_Hencl_Maly_2023} seem to indicate that if the convex function \(H\) satisfies stronger coercivity properties then the compactness of the minimisation space could be restored.

The paper is organised as follows. We start in Section \ref{sec:preliminaries} by recalling some definitions and preliminary results. In Section \ref{sec:compactness} we prove the compactness of sequences of maps in \(\B\) with a uniform bound on the neo-Hookean energy and on the $BV$ norm of their inverses. Section \ref{sec:lower_semicontinuity} is devoted to the proofs of the lower semicontinuity of \(F\) in \(\B\) and of Theorem \ref{main theorem}.

\subsection{Acknowledgements } 

D.~Henao was supported by FONDECYT grant N.~1231401 and by
Center for Mathematical Modeling, FB210005, Basal ANID Chile.

C. Mora-Corral has been supported by the Spanish Agencia Estatal de Investigaci\'on through projects  PID2021-124195NB-C32, the Severo Ochoa Programme CEX2019-000904-S, the ERC Advanced Grant 834728 and by the Madrid Government (Comunidad de Madrid, Spain) under the multiannual Agreement with UAM in the line for the Excellence of the University Research Staff in the context of the V PRICIT (Regional Programme of Research and Technological Innovation).

The research of R. Rodiac is part of the project No.\ 2021/43/P/ST1/01501 co-funded by the National Science Centre and the European Union Framework Programme for Research and Innovation Horizon 2020 under the Marie Skłodowska-Curie grant agreement No.\ 945339. For the purpose of Open Access, the author has applied a CC-BY public copyright licence to any Author Accepted Manuscript (AAM) version arising from this submission.  


\section{Notation and preliminaries}\label{sec:preliminaries}

Let \(U\) be an open set of \(\R^3\). For a vectorial map \(\vec u: U \rightarrow \R^3\) we denote by \(D \vec u\) its distributional Jacobian matrix. When \(\vec u\) is in \(BV(U,\R^3)\) we let \( D \vec u=D^a \vec u+D^s \vec u=D^a \vec u+D^j\vec u+D^c\vec u\) be the standard decomposition of $D \vec u$, where \(D^a \vec u\) denotes the absolutely continuous part of \( D \vec u\) with respect to the Lebesgue measure, and \(D^s \vec u\) denotes its orthogonal part. It is itself divided into the jump part \( D^j\vec u\) and the Cantor part \(D^c \vec u\). 
We will also use the notion of approximate differentiability (see, e.g., \cite[Section 3.1.2]{Federer69}, \cite[Definition 2.3]{MuSp95} or \cite[Section\ 2.3]{HeMo12}); if \(\vec u: U \rightarrow \R^3\) is approximately differentiable we denote by \(\nabla \vec u\) its approximate differential.
Due to the Calder\'on-Zygmund theorem, every \(\vec u\in BV(\Om,\R^3)\) is approximately differentiable a.e.\ and \(D^a \vec u=\nabla\vec u \, \mc{L}^3 \). In particular, with a small abuse of notation, for Sobolev maps \(D \vec u = \nabla \vec u \) a.e.
The same notation applies to a scalar function $\phi: U \to \R$.

The Lebesgue measure of a measurable set \(A\subset \R^3\) is denoted by \(|A|\). We say that two sets \(A,B\) are equal a.e.\ and we write \(A=B\) a.e.\ if \(|A\setminus B|=|B\setminus A|=0\). Given a measurable set \(A \subset \R^3\) and a point \(\vec x \in \R^3\), we define the density of \(A\) at \(\vec x\) by
\begin{equation}\label{eq:densities}
D(A,\vec x):=\lim_{r \to 0} \frac{|B(\vec x,r)\cap A|}{|B(\vec x,r)|}
\end{equation}
when the limit exists.
Here $B(\vec x,r)$ is the open ball centred at $\vec x$ of radius $r$.

The set of \(3\times 3\) matrices with coefficients in \(\R\) is denoted by \(\R^{3\times 3}\), while \(\R^{3\times 3}_+\) is its subset of matrices with positive determinant.
The adjoint and cofactor of \( \vec A\in \R^{3\times 3}\) are denoted by \(\adj \vec A\) and \( \cof \vec A \), respectively, so that \(\vec A(\adj \vec A)=(\adj \vec A) \vec A=(\det \vec A) \operatorname{Id} \) and \( \cof \vec A =(\adj \vec A)^T\).

We recall the area formula of Federer (\cite[Proposition  2.6]{MuSp95} and \cite[Theorem 3.2.5 and Theorem 3.2.3]{Federer69}). We will use the notation $\mc{N}(\vec u, A,\vec y)$ for the number of preimages of a point $\vec y$ in the set $A$ under~$\vec u$. In this section, \(\Om\) is any bounded domain of \(\R^3\).

\begin{proposition}\label{prop:area-formula}
Let $\vec u$ be approximately differentiable a.e.\ in \(\Om\), and denote the set of approximate differentiability points of $\vec u$ by $\Om_d$. Then, for any measurable set $A\subset \Om$ and any measurable function $\varphi:\R^3 \rightarrow \R$,
\begin{equation*}
\int_A (\varphi \circ \vec u) \left| \det \nabla \vec u \right| \dd \vec x =\int_{\R^3} \varphi(\vec y) \, \mc{N}(\vec u,\Om_d\cap A,\vec y) \, \dd \vec y
\end{equation*}
whenever either integral exists. Moreover, if a map $\psi:A\rightarrow \R$ is measurable and $\bar{\psi}:\vec u(\Om_d\cap A) \rightarrow \R$ is given by
\begin{equation*}
\bar{\psi}(\vec y):= \sum_{\vec x \in \Om_d \cap A, \ \vec u(\vec x)
= \vec y} \psi(\vec x) , \quad \vec y \in \vec u(\Om_d\cap A)
\end{equation*}
then $\bar{\psi}$ is measurable and
\begin{equation}\label{eq:area-formula}
\int_A\psi(\varphi\circ \vec u) \left| \det \nabla \vec u \right| \dd \vec x= \int_{\vec u(\Om_d\cap A)} \bar{\psi} \, \varphi \, \dd \vec y,
\end{equation}
whenever the integral on the left-hand side of \eqref{eq:area-formula} exists.
\end{proposition}
We observe that the previous proposition implies that, for \( \vec u\) approximately differentiable a.e., 
\begin{equation}\label{eq:null_set_property}
|\vec u (N\cap \Om_d)|=0 \quad \text{ whenever } |N|=0.
\end{equation}

We will need to work with a set of points satisfying more properties than just approximate differentiability.

\begin{definition}\label{def:Om0}
Let $\vec u$ be approximately differentiable a.e.\ and such that $\det D \vec u\neq 0$ a.e. We define $\Om_0$ as the set of $\vec x\in \Om$ for which the following are satisfied:
\begin{enumerate}
\item the approximate differential of $\vec u$ at $\vec x$ exists and equals $D \vec u(\vec x)$,
\item there exist $\vec w\in C^1(\R^3,\R^3)$ and a compact set $K \subset \Om$ of density $1$ at $\vec x$ such that $\vec u|_{K}=\vec w|_{K}$ and $\nabla \vec u|_{K}=D \vec w|_{K}$,
\item $\det \nabla \vec u(\vec x)\neq 0$.
\end{enumerate}
\end{definition}
It can be seen from \cite[Theorem 3.1.8]{Federer69}, 
Rademacher's Theorem and Whitney's Theorem that \( \Om_0\) is a set of full Lebesgue measure in \( \Om \), i.e., \( |\Om \setminus \Om_0|=0\). 

\begin{definition}\label{def:geometric_image}
For any measurable set $A$ of $\Om$, the geometric image of $A$ under an a.e.\ approximately differentiable map $\vec u$ is defined by  
\begin{equation*}
\imG(\vec u,A) : =\vec u(A\cap \Om_0) ,
\end{equation*}
with \(\Om_0\) as in Definition \ref{def:Om0}.
\end{definition}

We will need the following result.
\begin{lemma}\label{lem:Muller_Spector}(\cite[Lemma 2.5]{MuSp95})
Let \(\vec u:\Om \rightarrow \R^3\) be approximately differentiable in almost all \(\Om\) and suppose that \(\det \nabla\vec u(\vec x)\neq 0\) for almost every \(\vec x \in \Om\). Let \(\Om_0\)  be as in Definition \ref{def:Om0}. Then for every \(\vec x\in \Om_0\) and every measurable set \(A\subset \Om\),
\begin{equation*}
D(\imG(\vec u,A),\vec u(\vec x))=1 \text{ whenever } D(A,\vec x)=1,
\end{equation*}
where the density is defined in \eqref{eq:densities}.
\end{lemma}

In order to define the inverse of maps which are approximately differentiable, one-to-one a.e.\ and such that \(\det \nabla\vec u\neq 0\) a.e., we first give the following lemma.

\begin{lemma}\label{lem:one-to-one}(\cite[Lemma 3]{HeMo10})
Let \(\vec u:\Om \rightarrow \R^3\) be approximately differentiable in almost all \(\Om\), one-to-one a.e., and suppose that \(\det \nabla\vec u(\vec x)\neq 0\) for a.e.\ \(\vec x\in \Om\). Let \(\Om_0\) be as in Definition \ref{def:Om0}. Then \( \vec u|_{\Om_0}\) is one-to-one.
\end{lemma}

\begin{definition}\label{def:inverses}
Let \(\vec u:\Om \rightarrow \R^3\) be approximately differentiable in almost all \(\Om\), one-to-one a.e., and suppose that \(\det D\vec u(\vec x)\neq 0\) for a.e.\ \(\vec x\in \Om\). Let \(\Om_0\) be as in Definition \ref{def:Om0}. Then we define the inverse \(\vec u^{-1}\) as the map \(\vec u^{-1}:\imG(\vec u,\Om)\rightarrow \R^3\) that sends every \(\vec y \in \imG(\vec u,\Om)\) to the only \(\vec x\in \Om_0\) such that \(\vec u(\vec x)=\vec y.\)
\end{definition}

Proposition \ref{prop:area-formula} will be used in the following form.

\begin{corollary}\label{cor:area-formula}
Let $\vec u \in \B$.
Then, for any measurable function $\varphi: \Om_{\vec b} \rightarrow \R$,
\begin{equation*}
\int_{\Om} (\varphi \circ \vec u) |\det \nabla \vec u |\, \dd \vec x =\int_{\Om_{\vec b}} \varphi \, \dd \vec y
\end{equation*}
whenever either integral exists. Moreover, if a map $\psi: \Om \rightarrow \R$ is measurable, then
\begin{equation*}
\int_{\Om} \psi(\varphi\circ \vec u) |\det \nabla \vec u |\, \dd \vec x= \int_{\Om_{\vec b}} ( \psi \circ \vec u^{-1} ) \, \varphi \, \dd \vec y ,
\end{equation*}
whenever the integral on the left-hand side exists.
\end{corollary}

\begin{proof}
Let us check that this is a particular case of Proposition \ref{prop:area-formula}.
Any $\vec u \in \B$, being Sobolev, is approximately differentiable a.e.
Since $\det D \vec u\neq 0$ a.e.\ and $\Om_{\vec b} = \imG(\vec u,\Om)$ a.e., the set 
\begin{equation*}
 \Om' := \{ \vec x \in \Om_0 : \det \nabla \vec u(\vec x) \neq 0 , \, \vec u (\vec x) \in \Om_{\vec b} \}
\end{equation*}
is of full measure in $\Om$; indeed, this is because $\vec u$ satisfies Lusin's $N^{-1}$ condition (i.e., the preimage by $\vec u$ of a set of measure zero has measure zero), being this a consequence of Proposition \ref{prop:area-formula} and $|\det \nabla \vec u(\vec x)|> 0$ a.e.\  (see, if necessary, \cite[Remark 2.3 (b)]{BrFrMo24}).
According to Lemma \ref{lem:one-to-one} and Definition \ref{def:inverses}, the inverse \(\vec u^{-1}:\imG(\vec u,\Om)\rightarrow \R^3\) is defined.
Denote by $\varphi_0$ the extension of $\varphi$ to $\R^3$ by zero.
By Proposition \ref{prop:area-formula},
\begin{equation*}
 \int_{\Om'} (\varphi_0 \circ \vec u) \left| \det \nabla \vec u \right| \dd \vec x 
 =\int_{\R^3} \varphi_0(\vec y) \, \mc{N}(\vec u, \Om',\vec y) \, \dd \vec y .
\end{equation*}
Now,
\begin{equation*}
\int_{\Om} (\varphi_0 \circ \vec u) \left| \det \nabla \vec u \right| \dd \vec x = \int_{\Om'} (\varphi_0 \circ \vec u) \left| \det \nabla \vec u \right| \dd \vec x 
\end{equation*}
and, by Lemma \ref{lem:one-to-one},
\begin{equation*}
 \int_{\R^3} \varphi_0(\vec y) \, \mc{N}(\vec u, \Om',\vec y) \, \dd \vec y = \int_{\Om_{\vec b}} \varphi (\vec y) \, \mc{N}(\vec u,\Om',\vec y) \, \dd \vec y  = \int_{\imG (\vec u, \Om)} \varphi (\vec y) \, \dd \vec y .
\end{equation*}
This proves the first equality of the statement.
The second one is analogous.
\end{proof}

\begin{proposition}\label{prop:relation_inverses}
Let \(\vec u\in H^1(\Om,\R^3)\) be such that \(\det D\vec u(\vec x)\neq 0\) for a.e.\ \(\vec x\in \Om\) and \(\vec u\) is one-to-one a.e. Let \(\Om_0\) be as in Definition \ref{def:Om0}. Then \(\vec u^{-1}\) is approximately differentiable at every \(\vec y \in \imG(\vec u,\Om)=\vec u(\Om_0)\) and its approximate differential satisfies 
\begin{equation}\label{eq:relation_inverses}
\nabla \vec u^{-1}\big ( \vec u(\vec x)\big ) 
 = D\vec u(\vec x) ^{-1}= \frac{\adj D \vec u (\vec x)}{\det D\vec u(\vec x)} \text{  for every } \vec x \in \Omega_0 .
\end{equation}
In particular, if we assume that \(\imG(\vec u,\Om)=\Om_{\vec b} \) a.e.\ then \(\nabla \vec u^{-1}\in L^1(\Om_{\vec b})\) with
\begin{equation}\label{eq:controlDau-1}
 \| \nabla \vec u^{-1} \|_{L^1(\Om_{\vec b})}\leq \frac{1}{\sqrt{3}}\| D\vec u\|^2_{L^2(\Om)}.
\end{equation} If we assume furthermore that \(\vec u^{-1}\in BV(\Om_{\vec b},\R^3)\) then \( D^a \vec u^{-1}(\vec u(\vec x)) = D\vec u(\vec x) ^{-1}\) for a.e.\ \(\vec x\in \Om\).
\end{proposition}

\begin{proof}
The proof is adapted from \cite[Theorem 2 iii)]{HeMo11}.
  Let \(\vec x_0\in \Om_0\) and define \(\vec y_0:=\vec u(\vec x_0)\) and \(\vec F:=D \vec u(\vec x_0)\). Thanks to Definition \ref{def:Om0}, \(\vec F\) is invertible and thanks to Lemma \ref{lem:Muller_Spector} we have \( D(\imG(\vec u,\Om),\vec y_0)=1\). Define, for each \(\delta >0\),
\begin{equation*}
E_\delta :=\left\{ \vec x\in \Om_0 \setminus \{ \vec x_0\}: \frac{|\vec u(\vec x)-\vec u(\vec x_0)-\vec F (\vec x-\vec x_0)|}{|\vec x-\vec x_0|}<\delta \right\}.
\end{equation*}
Since \(\vec u\) is approximately differentiable at \(\vec x_0\) and the set \(\Om_0\) is of full measure in \(\Om\), we deduce that \( D(E_\delta, \vec x_0)=1\) for all \(\delta>0\). Now for each \(\e>0\) we set
\begin{equation*}
A_\e:=\left\{ \vec y \in \imG(\vec u,\Om)\setminus \{\vec u (\vec x_0)\}: \frac{|\vec u^{-1}(\vec y)-\vec x_0-D\vec u(\vec x_0)^{-1} (\vec y-\vec u(\vec x_0))|}{|\vec y- \vec u(\vec x_0)|}>\e\right\}.
\end{equation*}
Let \(\vec x \in \Om_0 \setminus \{ \vec x_0\}\) and \(\vec y:=\vec u(\vec x)\). Thanks to Lemma \ref{lem:one-to-one}, we have \(\vec y \neq \vec y_0\). Set \( \vec r:=\vec y-\vec y_0-\vec F(\vec x-\vec x_0)\). Then
\[
\frac{|\vec x-\vec x_0-\vec F^{-1}(\vec y-\vec y_0)|}{|\vec y-\vec y_0|} \leq |\vec F^{-1}|\frac{|\vec r|}{|\vec x-\vec x_0|}\frac{|\vec x-\vec x_0|}{|\vec y-\vec y_0|}
\leq |\vec F^{-1}| \frac{|\vec r|}{|\vec x-\vec x_0|} \frac{1}{\left|\vec F \frac{\vec x-\vec x_0}{|\vec x-\vec x_0|} \right|-\frac{|\vec r|}{|\vec x-\vec x_0|}}.
\]
This shows that if \(\vec u(E_\delta)\cap A_\e\neq \varnothing\) for some \(\delta,\e>0\), then 
\begin{equation}\label{eq:(30)}
\e>|\vec F^{-1}|\frac{\delta}{\inf\{|\vec F \vec v|:|\vec v|=1\}-\delta}.
\end{equation}
Fix \(\e>0\). Then there exists \(\delta>0\) such that \eqref{eq:(30)} does not hold and, hence, \(\vec u(E_\delta)\cap A_\e=\varnothing\). As \(D(E_\delta, \vec x_0)=1\), then by Lemma \ref{lem:Muller_Spector}, \(D(\R^3\setminus \vec u(E_\delta),\vec y_0)=0\), and, hence, \(D(A_\e,\vec y_0)=0\). This proves that \(\vec u^{-1}\) is approximately differentiable at \(\vec u(\vec x_0)\) and its approximate differential is equal to \(D\vec u(\vec x_0)^{-1}\) and, thus, \eqref{eq:relation_inverses} holds.

We now assume that \(\imG(\vec u,\Om)=\Om_{\vec b}\) a.e.  Propositions \ref{prop:area-formula} and \ref{prop:relation_inverses} as well as the matrix inequality
\begin{equation}\label{eq:sqrt3}
 |\vec A|^2 \geq \sqrt{3} \left| \cof \vec A \right|, \qquad \vec A \in \R^{3 \times 3}
\end{equation}
(see Lemma \ref{le:better than 2} \ref{item:sqrt3} in Section \ref{sec:lower_semicontinuity}) show that 
\[
\int_{\Om_{\vec b}} |\nabla \vec u^{-1}(\vec y)| \, \dd \vec y =\int_{\Om} |\cof D\vec u (\vec x)| \, \dd \vec x \leq \frac{1}{\sqrt{3}} \int_{\Om} |D\vec u(\vec x)|^2 \, \dd \vec x 
\]
(see, if necessary, the proof of Corollary \ref{cor:area-formula} to see how the area formula applies to this context).

Finally, if in addition \(\vec u^{-1}\in BV(\Om_{\vec b},\R^3)\), then \(D^a \vec u^{-1} = \nabla \vec u^{-1}\) a.e., which implies the conclusion.
\end{proof}

We now make explicit the failure of the divergence identities \eqref{eq:divergence_identities} by means of the functional $\E$ introduced in \cite{HeMo10}.

\begin{definition}\label{def:surface_energy}
Let $\vec u : \Om \to \R^3$ be measurable and approximately differentiable a.e\@.
Suppose that $\det  \nabla \vec u \in L^1_{\loc} (\Om)$ and $\cof \nabla \vec u \in L^1_{\loc} (\Om,\R^{3\times 3})$.
For every $\vec f \in C^1_c (\Om \times \R^3,\R^3)$, define
\begin{equation*}
 \E (\vec u, \vec f) := \int_{\Om} \left[ \cof \nabla \vec u (\vec x) \cdot D \vec f (\vec x, \vec u (\vec x)) 
 + \det \nabla\vec u (\vec x) \div \vec f (\vec x, \vec u (\vec x))  \right] \dd \vec x ,
\end{equation*}
where $D \vec f (\vec x, \vec y)$ denotes the derivative of $\vec f (\cdot, \vec y)$ 
evaluated at $\vec x$, while $\div \vec f (\vec x, \vec y)$ is the divergence of $\vec f (\vec x, \cdot)$ 
evaluated at $\vec y$.
\end{definition}

By definition of distributional divergence, the divergence identities hold if and only if $\E (\vec u, \phi \, \vec g) = 0$ for all $\phi \in C^1_c (\Om)$ and $\vec g \in C^1_c (\R^3, \R^3)$.
In fact, by density of sums of functions of separate variables (see, e.g., \cite[Corollary 1.6.5]{Llavona86}), this holds if and only if $\E (\vec u, \vec f) = 0$ for all $\vec f \in C^1_c (\Om \times \R^3,\R^3)$.
Of course, $\phi \, \vec g$ denotes the function $\phi (\vec x) \, \vec g (\vec y)$ for $(\vec x, \vec y) \in \Om \times \R^3$.

The following lemma helps understand our choice of the set $\B$ as the set to pose the relaxation of $F$.
  \begin{lemma}
    \label{le:two_facts}
    Let $\vec u\in H^1(\Omega, \R^3)$ be one-to-one a.e.~and satisfy
    $\vec u=\vec b$ in $\Omega \setminus \widetilde \Omega$,
    $\det D\vec u \in L^1_{\loc}(\Omega)$ and
    $\det D\vec u>0$ a.e.\
   Then the following conditions are equivalent:
   \begin{enumerate}[label=\alph*)]
    \item 
      $\vec u$ satisfies the divergence identities \eqref{eq:divergence_identities}.
    \item \label{it:Sobolev_inverse}
      $\imG(\vec u, \Omega) = \Omega_{\vec b}$ a.e.\ and
      $\vec u^{-1}\in W^{1,1}(\Omega_{\vec b}, \R^3)$.
   \end{enumerate}
  \end{lemma}

  \begin{proof}
``\eqref{eq:divergence_identities} implies $\imG(\vec u, \Omega) = \Omega_{\vec b}$ a.e.'':
    we refer to \cite[Proposition 4.11]{BaHEMoRo_23PUB} for  a proof in the axisymmetric setting. 
    In the general case the proof is exactly as in \cite[Theorem 4.1]{BaHeMo17} but using 
    the Brezis-Nirenberg degree, cf.\ \cite{Brezis_Nirenberg_1995,CoDeLe03}, instead of the 
    Brouwer degree; the conclusion is that there exists an open set \(U\) with \(\widetilde{\Om}\Subset U\Subset \Om\) 
    such that
\[ \imG(\vec u,U)=\imT(\vec u,U)=\imT(\vec b,U)=\vec b(U) \text{ a.e.}, \]
hence \(\imG(\vec u,\Om)=\Om_{\vec b}\).

``\eqref{eq:divergence_identities} implies $\vec u^{-1}\in W^{1,1}(\Omega_{\vec b}, \R^3)$'':
see \cite[Proposition 4.12]{BaHEMoRo_23PUB} in the axisymmetric case. In the general case, one can apply \cite[Lemma 5.1]{BaHeMo17} to show that condition INV holds; as before one uses the Brezis-Nirenberg degree. Then, with \cite[Theorem 3.4]{HeMo15} one concludes that \(\vec u^{-1}\) is Sobolev, first in some open set \(V\) with \(\widetilde{\Om}_\vec b\Subset V\Subset \Om_\vec b\) and then in \(\Om_\vec b\).

``\ref{it:Sobolev_inverse} implies \eqref{eq:divergence_identities}'':
 this follows, e.g., from Step 4 in the proof of \cite[Thm.~2]{HeMo11}, but we include here the short proof for the convenience of the reader.
 Suppose that  $\imG(\vec u, \Omega) = \Omega_{\vec b}$ a.e.
 By Definition \ref{def:surface_energy}, Corollary \ref{cor:area-formula}, Proposition  \ref{prop:relation_inverses} and the relation $\vec A^{-1}=\frac{\adj \vec A}{\det \vec A}$ valid for $\vec A \in \R^{3 \times 3}_+$, we have,
 for \(\phi \in C^\infty_c(\Om)\) and \(\vec g\in C^\infty_c(\R^3,\R^3)\),  that
\begin{align*}
\E(\vec u,\phi \, \vec g)&=\int_{\Om}\left[ \left( \cof D\vec u(\vec x) D \phi(\vec x) \right) \cdot \vec g(\vec u(\vec x))+\det D\vec u(\vec x)\phi(\vec x)\div \vec g(\vec u(\vec x)) \right] \dd \vec x\\
&=\int_{\Om_{\vec b}} \left[ \left( \nabla \vec u^{-1}(\vec y)^T D \phi (\vec u^{-1}(\vec y)) \right) \cdot \vec g(\vec y) +\phi(\vec u^{-1}(\vec y))\div \vec g (\vec y) \right] \dd \vec y.
\end{align*}
As $\phi \in C^1_c(\Om)$ we have \(\phi\circ \vec u^{-1} \in L^{\infty}(\Om_{\vec b})\) and, by \cite[Proposition 3.71]{AmFuPa00} and Proposition \ref{prop:relation_inverses}, that
\begin{equation*}
 \nabla (\phi\circ \vec u^{-1})= (\nabla \vec u^{-1})^T D \phi(\vec u^{-1}) ,
\end{equation*}
which is in \(L^1(\Om_{\vec b}) \) thanks to \eqref{eq:controlDau-1}.
We can then write
\begin{equation}\label{eq:relation_E_derivative_BV}
\E(\vec u, \phi \, \vec g) = \langle \phi \circ \vec u^{-1},\div \vec g \rangle_{\D'(\Om_{\vec b})} + \langle \nabla (\phi \circ \vec u^{-1}), \vec g \rangle_{\D'(\Om_{\vec b},\R^3)}.
\end{equation}
  If $\vec u^{-1}\in W^{1,1}(\Omega_{\vec b}, \R^3)$
  then, by the chain rule (e.g., \cite[Theorem 4.2.2.4]{EvGa92}),
  $\phi \circ \vec u^{-1} \in W^{1,1} (\Om)$ and the expression in \eqref{eq:relation_E_derivative_BV} is zero for every $\vec g \in C^1_c(\R^3, \R^3)$.
  \end{proof}

\section{Compactness of sequences in $\B$ with equibounded energy $F$}\label{sec:compactness}

In this section we show that the set of deformation maps such that their geometric image is  \(\Om_\vec b\) and whose inverses are in \( BV(\Om_{\vec b},\R^3)\) is compact for the weak convergence in \(H^1\) if we assume a uniform bound on the neo-Hookean energy and on the $BV$ norm of the inverses.
Furthermore, those bounds also provide that the weak \(H^1\) limit is one-to-one a.e.\ and satisfies that \(\det D \vec u\neq 0\) a.e.
In this respect, a uniform bound on the $BV$ norm of the inverses of a sequence of deformation maps plays a role analogous to the one of a uniform bound on the surface energy defined in \cite{HeMo10}.
An intermediate step is to show the validity of the divergence identities for $\vec u^{-1}$.

We start with the following variant of \cite[Theorem 2]{HeMo10}.

\begin{proposition}\label{pr:limituj}
Let $\{\vec u_j\}_j$ be a sequence in $\B$. Assume that $\{F(\vec u_j)\}_j$ is equibounded and that
$\vec u_j \rightharpoonup \vec u$ in $H^1(\Om,\R^3)$. Then   
 \begin{enumerate}[label=\roman*)]
 \item \label{i)} $\det D \vec u \ne 0$\ \text{a.e.}
 \item \label{ii)} $\vec u$ is one-to-one a.e.
 \item \label{iii)}$\imG(\vec u, \Omega) = \Om_{\vec b}$\ a.e.
 \item \label{iv)}$\vec u^{-1} \in BV(\Om_{\vec b}, \R^3)$,
 \item \label{v)} up to a subsequence, $\vec u_j^{-1} \weakcs \vec u^{-1}$ in $BV (\Om_{\vec b}, \R^3)$ and $\vec u_j^{-1}\to \vec u^{-1}$ a.e. 
 \item \label{vi)} $  |\det D \vec u_j| \rightharpoonup |\det D \vec u|$ in \(L^1(\Om_{\vec b})\).
 \end{enumerate}
\end{proposition}

\begin{proof}
Since $\sup_j \int_\Om H(|\det D\vec u_j|) <\infty$, by using the De la Vall\'ee Poussin criterion we can find 
$\theta \in L^1(\Om)$ such that, for a subsequence,
\begin{equation*}
|\det D\vec u_j| \rightharpoonup \theta \text{ in } L^1(\Om).
\end{equation*}
As $|\det D\vec u_j| > 0$ a.e.\ for all $j \in \N$, we have that $\theta \geq 0$ a.e.
In fact, $\theta>0$ a.e., since otherwise there would exist a set $A \subset \Om$ of positive measure such that, for a subsequence, $\det D\vec u_j \to 0$ in $L^1 (A)$ and a.e.\ in $A$.
By Fatou's lemma and the properties of \(H\) in \eqref{eq:explosiveH}, we would obtain
\[
 F(\vec u_j) \geq \int_{\Om} H(|\det D \vec u_j|) \, \dd \vec x \geq \int_A H(|\det D \vec u_j|) \, \dd \vec x \to \infty \qquad \text{as } j \to \infty ,
\]
a contradiction.
Hence, $\theta>0$ a.e.\ in $\Om$.
Passing to a subsequence we can also assume that $\vec u_j\to \vec u$ a.e.

We first want to show that the Jacobian determinant of $\vec u$ is different from zero.
Let $\psi: \Om_{\vec b} \to \R$ be continuous and bounded.
An application of the change of variables formula (Corollary \ref{cor:area-formula}) shows that 
 \begin{equation*}
  \int_\Om \psi(\vec u_j(\vec x)) |\det D \vec u_j (\vec x) |\, \dd \vec x = 
  \int_{\Om_\vec b} \psi(\vec y) \, \dd \vec y.
 \end{equation*}
 Since \(\{\psi(\vec u_j)\}_j\) is equibounded in \(L^\infty\), a standard convergence result (see, e.g., \cite[Proposition 2.61]{FoLe07}) shows that
 \begin{equation}\label{eq:psitheta}
  \int_\Om \psi(\vec u(\vec x)) \, \theta (\vec x)\, \dd \vec x = \int_{\Om_{\vec b}} \psi(\vec y) \, \dd \vec y .
 \end{equation}
By approximation, the above formula remains valid for any $\psi$ bounded Borel.
Let $V:=\{\vec x\in \Om_d : \det D \vec u(\vec x)=0\}$.
Proposition \ref{prop:area-formula} shows that $|\vec u(V)|=0$.
Let $U$ be a Borel set such that $\vec u(V)\subset U$
and $|U|=0$. By taking  $\psi=\chi_U$, we obtain 
 \begin{equation*}
  \int_V \theta(\vec x) \, \dd\vec x = 0.
 \end{equation*}
Since $\theta$ is positive a.e., necessarily $|V|=0$, proving \ref{i)}.
 
Set $\vec v_j :=\vec u_j^{-1}$.
Thanks to \eqref{eq:controlDau-1}, from the assumption that \( \{ F(\vec u_j) \}_j\) is equibounded, we find that both the absolutely continuous and the singular parts of $D\vec v_j$ are equibounded, so $\|D\vec v_j\|$ is equibounded. By compactness, up to a subsequence, 
$\vec v_j$ converges weakly\(^*\) and a.e.\ in $BV({\Om}_{\vec b},\R^3)$ to some $\vec v$. 
 
Now we apply a similar argument leading to \eqref{eq:psitheta}.
For $\varphi :\Om \to \R$ and \( \psi : \Om_{\vec b} \to \R \) both continuous and bounded, we obtain first 
 \begin{equation*}
  \int_\Om \varphi(\vec x) \psi(\vec u_j(\vec x))|\det D \vec u_j(\vec x)| \, \dd\vec x 
  = \int_{\Om_{\vec b}} \varphi(\vec v_j(\vec y)) \psi(\vec y) \, \dd\vec y ,
 \end{equation*}
then
\begin{equation*}    
  \int_\Om \varphi(\vec x) \psi(\vec u(\vec x))\theta(\vec x) \, \dd\vec x 
  = \int_{\Om_{\vec b}} \varphi(\vec v(\vec y)) \psi(\vec y) \, \dd\vec y
 \end{equation*}
and, finally, that the above formula is valid for all $\varphi$ and $\psi$ bounded Borel.
As a consequence, with \(\Om_d\) the set of approximate differentiability of \(\vec u\), Proposition \ref{prop:area-formula} shows that
 \begin{equation*}
  \int_{\R^3} \psi(\vec y) \sum_{\substack{\vec x \in \Om_d\\
    \vec u(\vec x) = \vec y}} \varphi(\vec x)\frac{\theta(\vec x)}{|\det D \vec u(\vec x)|}\dd\vec y 
    = \int_{\Om_{\vec b}} \varphi(\vec v(\vec y)) \psi(\vec y) \dd\vec y
 \end{equation*}
 for any $\varphi$ and $\psi$ bounded Borel.
Since the formula holds for all $\psi$, then  
 \begin{equation*}
   \sum_{\substack{\vec x \in \Om_d\\
    \vec u(\vec x) = \vec y}} \varphi(\vec x)\frac{\theta(\vec x)}{|\det D \vec u(\vec x)|} = 0
    \quad \text{for a.e.\ $\vec y\in \R^3\setminus \Om_{\vec b}$}
 \end{equation*}
and 
 \begin{equation}\label{eq:second}
   \sum_{\substack{\vec x \in \Omega_d\\
    \vec u(\vec x) = \vec y}} \varphi(\vec x)\frac{\theta(\vec x)}{|\det D \vec u(\vec x)|} =  \varphi(\vec v(\vec y))
    \quad \text{for a.e.\ $\vec y\in \Om_{\vec b}$}.
 \end{equation}
From the first identity we obtain $\vec u(\Omega_d) \subset \Om_{\vec b}$ a.e., because
\begin{equation*}
 \int_{\{\vec x \in \Omega_d: \vec u(\vec x) \not\in \Om_{\vec b}\}} \varphi(\vec x)\theta(\vec x) \, \dd \vec x =
 \int_{\vec u(\Omega_d)\setminus \Om_{\vec b}} \sum_{\substack{\vec x \in \Omega_d\\
    \vec u(\vec x) = \vec y}} \varphi(\vec x)\frac{\theta(\vec x)}{|\det D \vec u(\vec x)|}\dd\vec y =0 ,
\end{equation*}
so taking $\varphi=\chi_\Omega$ and using the positivity of $\theta$ we obtain $|\{\vec x \in \Omega_d: \vec u(\vec x) \not\in \Om_{\vec b}\}| = 0$, and, by \eqref{eq:null_set_property}, $|\vec u(\Omega_d)\setminus \Om_{\vec b}| = 0$. 

Now consider the second identity. 
Any point in $\Om_{\vec b}$ for which formula \eqref{eq:second} is valid is necessarily in $\vec u(\Omega_d)$, since otherwise the left-hand side would always be zero, whereas one may find $\varphi$ such that the right-hand side is not zero.
This shows that $|\Om_{\vec b} \setminus \vec u (\Om_d)| = 0$ and, hence, $\vec u (\Om_d) = \Om_{\vec b}$ a.e.
By property \eqref{eq:null_set_property}, $\vec u (\Om_d) = \vec u (\Om_0)$ a.e., so \ref{iii)} is proved.

Let $\Om_{\vec b}'$ be the set of $\vec y \in \Om_{\vec b}$ for which formula \eqref{eq:second} is valid.
Fix $\vec y_0 \in \Om_{\vec b}'$ and set $\vec x_0 := \vec v(\vec y_0)$.
Taking any bounded Borel $\varphi$ such that $\varphi (\vec x_0) = 0$ we obtain from \eqref{eq:second} that
\[
\sum_{\substack{\vec x \in \Omega_d\\
    \vec u(\vec x) = \vec y_0}} \varphi(\vec x)\frac{\theta(\vec x)}{|\det D \vec u(\vec x)|} =  0 .
\]
As there is no composition with $\vec u$ or $\vec v$ in the formula above, by approximation it remains valid for any bounded measurable $\varphi$ with $\varphi (\vec x_0) = 0$.
Taking $\varphi$ to be the characteristic function of $\{ \vec x \in \Om_d \setminus \{ \vec x_0 \} : \theta (\vec x) > 0 \}$ we obtain
\[
\sum_{\substack{\vec x \in \Omega_d \setminus \{ \vec x_0 \}\\
    \vec u(\vec x) = \vec y_0 , \ \theta (\vec x) > 0}} \frac{\theta(\vec x)}{|\det D \vec u(\vec x)|} =  0 ,
\]
so there is no $\vec x \in \Omega_d \setminus \{ \vec x_0 \}$ such that $\vec u(\vec x) = \vec y_0$ and $\theta (\vec x) > 0$.
This shows that $\vec u$ is one-to-one in $\{ \vec x \in \Om_d : \vec u (\vec x) \in \Om_{\vec b}' , \, \theta (\vec x) > 0\}$.
Thanks to \ref{i)}, this set is of full measure in $\Om$ (see, e.g., \cite[Remark 2.3 (b)]{BrFrMo24} for a proof, 
which in fact is an easy consequence of Proposition \ref{prop:area-formula}).
This proves \ref{ii)}.

Now we take $\varphi = \chi_{\{ \vec x_0 \}}$ in \eqref{eq:second} and obtain that
\[
 \sum_{\substack{\vec x \in \Omega_d \cap \{ \vec x_0 \}\\ \vec u(\vec x) = \vec y}} \frac{\theta(\vec x)}{|\det D \vec u(\vec x)|} = 1 ,
\]
which implies that $\vec x_0 \in \Om_d$, $\vec u (\vec x_0) = \vec y_0$ and $\theta(\vec x_0) = |\det D \vec u(\vec x_0)|$.
Equality $\vec u (\vec x_0) = \vec y_0$ says that $\vec v= \vec u^{-1}$ in $\Om_{\vec b}'$ and, hence, $\vec v= \vec u^{-1}$ a.e.
This proves \ref{iv)} and \ref{v)}.
Finally, equality $\theta(\vec x_0) = |\det D \vec u(\vec x_0)|$ says that $\theta = |\det D \vec u|$ in the set $\{ \vec x \in \Om_d: \vec u (\vec x) \in \Om_{\vec b}' \}$, which, as before, has full measure in $\Om$.
Thus, $\theta = |\det D \vec u|$ a.e.\ and, hence, \ref{vi)} holds. 
\end{proof}

As an immediate consequence of Proposition \ref{pr:limituj}, we obtain  the compactness of sequences in $\B$ with equibounded energy $F$.

\begin{corollary}\label{cor:Bcompact}
Let $\{\vec u_j \}_j \subset \B$ be such that $\{ F (\vec u_j)\}_j$ is equibounded.
Then there exists $\vec u \in \B$ such that, up to a subsequence, $\vec u_j \rightharpoonup \vec u$ 
in $H^1 (\Om, \R^3)$. In particular $\overline{\A}\subset \B$.
\end{corollary}

\section{Lower semicontinuity of \(F\) in \(\B\) and proof of the theorem}\label{sec:lower_semicontinuity}

The key argument to prove the lower semicontinuity of the energy \(F\) is to use a change of variables to express the neo-Hookean energy in terms of the inverse deformations. This allows to control from below the Dirichlet part of the neo-Hookean energy by the \(L^1\) norm of the absolutely continuous part of the  gradient of the inverses with the optimal constant \(2\). We then use the theory of $BV$ functions to obtain the desired lower semicontinuity.

We first start by studying the weak convergence in \(L^1\) of \( \{ \sgn(\det \nabla \vec u_j^{-1})\cof \nabla \vec u_j^{-1} \}_j\) and \( \{ |\det \nabla \vec u_j^{-1}| \}_j \).

\begin{proposition}\label{pr:convergenza det e cof}
Let $\{\vec u_j\}_j$ be a sequence in $\B$. Assume that $\{ F (\vec u_j)\}_j$ is equibounded and that
$\vec u_j \rightharpoonup \vec u$ in $H^1(\Om,\R^3)$.
Then, up to a subsequence, $|\det \nabla \vec u_j^{-1}| \rightharpoonup |\det \nabla \vec u^{-1}|$ in $L^1 (\Om_{\vec b})$ and
\begin{equation*}
\sgn(\det \nabla \vec u_j^{-1})\cof \nabla \vec u^{-1}_j \rightharpoonup \sgn(\det \nabla \vec u^{-1})\cof \nabla \vec u^{-1}
\text{ in } L^1({\Om}_{\vec b},\R^{3\times 3}).
\end{equation*} 
\end{proposition}

\begin{proof}
The proof is divided into two steps.

\emph{Step 1: Equiintegrability.}
We first show that $\{ |\det \nabla\vec u_j^{-1}| \}_j$ is equiintegrable.
By Corollary \ref{cor:area-formula} and Proposition \ref{prop:relation_inverses},
\begin{equation*}
\int_\Om H(|\det D \vec  u_j(\vec x)|) \, \dd \vec x
 =\int_{\Om_{\vec b}} H \biggl( \biggl |\frac{1}{\det \nabla \vec u_j^{-1} (\vec y)} \biggr |\biggr) |\det \nabla \vec u_j^{-1}(\vec y)|\, \dd \vec y.
\end{equation*}
Since \( H_1(t):= H(1/t)t\) is convex and satisfies \eqref{eq:explosiveH}, the De la Vall\'ee Poussin criterion shows that $\{ |\det \nabla\vec u_j^{-1} |\}_j$ is equiintegrable, as $\{ F (\vec u_j) \}_j$ is bounded.
 
 Let $V\subset\Om_{\vec b}$ be a Borel set. On the one hand, using Corollary \ref{cor:area-formula} and formula $\adj \vec A^{-1} \det \vec A = \vec A$ valid for $A\in \R^{3 \times 3}_+$, we get
\begin{equation}\label{eq:adju-1}
    \int_V |\adj \nabla \vec u^{-1}_j(\vec y)| \, \dd\vec y = 
    \int_{\vec u^{-1}_j(V)} |D\vec u_j(\vec x)| \, \dd\vec x.
 \end{equation}
 On the other hand,
 \begin{equation*}
  |\vec u^{-1}_j(V)| = \int_V |\det \nabla \vec u^{-1}_j(\vec y) |\, \dd\vec y.
 \end{equation*}
 Take a fixed $\e>0$. Since the sequence $\{D\vec u_j\}_j$ is equiintegrable, there exists $\delta>0$ such that for any measurable 
 $A\subset \Omega$ with $|A|<\delta$ we have that for all $j\in \N$ 
 \begin{equation*}
    \int_A |D\vec u_j(\vec x)| \, \dd \vec x < \e.
 \end{equation*}
 We apply this to $A=\vec u^{-1}_j(V)$, complementing it with the equiintegrability of $\{ |\det \nabla\vec u^{-1}_j| \}_j$, 
 which gives that there exists $\eta>0$ such that 
 \begin{equation*}
  \text{if}\ |V| < \eta \ \text{then}\ \int_V |\det \nabla \vec u^{-1}_j(\vec y)| \, \dd\vec y < \delta.
 \end{equation*}
This shows that the sequence $\{ |\adj \nabla \vec u^{-1}_j| \}_j$ is equiintegrable.

\smallskip
\emph{Step 2: Weak$^*$ convergence in the sense of measures.}
By Corollary \ref{cor:Bcompact} the limit $\vec u$ belongs to $\B$. 
Let $\psi \in C (\R^3)$ be bounded. Up to a subsequence, \( \vec u_j\rightarrow \vec u\) a.e.\ Hence by the dominated 
convergence theorem $\psi \circ \vec u_j \to \psi \circ \vec u$ in $L^q$ for all $q < \infty$. 
Then, by Corollary \ref{cor:area-formula} and Proposition \ref{prop:relation_inverses}, we get
\begin{equation*}
\int_{\Om_{\vec b}} |\det \nabla\vec u^{-1}_j (\vec y)| \psi ( \vec y) \, \dd\vec y 
= \int_\Omega \psi(\vec u_j(\vec x))\, \dd \vec x\to
\int_\Omega \psi(\vec u(\vec x))\, \dd \vec x
=\int_{\Om_{\vec b}} |\det \nabla\vec u^{-1}(\vec y)| \psi ( \vec y) \, \dd\vec y.
\end{equation*}
In the right-hand side we used the fact that $\vec u\in\B$ to return to the deformed configuration.
This proves the convergence of  $|\det \nabla \vec u_j^{-1}|$ in the sense of measures, and then
in $L^1$ due to the equiintegrability. 

Similarly, let $\vec\psi \in C (\R^3, \R^{3 \times 3})$ be bounded. 
As in \eqref{eq:adju-1},
\begin{equation}\label{eq:adju-1bis}
\int_{\Om_{\vec b}} \sgn(\det \nabla \vec u_j^{-1}(\vec y))\adj \nabla \vec u^{-1}_j (\vec y) \cdot \vec \psi ( \vec y) \, \dd\vec y 
= \int_\Omega D\vec u_j (\vec x)\cdot \vec \psi(\vec u_j(\vec x))\, \dd \vec x.
\end{equation}
Using the fact that $\vec \psi \circ \vec u_j \to \vec\psi \circ \vec u$ in $L^2$, we have
\[
 \int_\Omega D\vec u_j (\vec x)\cdot \vec \psi(\vec u_j(\vec x))\, \dd \vec x \to \int_\Omega D\vec u (\vec x)\cdot \vec \psi(\vec u(\vec x))\, \dd \vec x .
\]
Since $\vec u\in\B$, formula \eqref{eq:adju-1bis} also holds for $\vec u$ in place of $\vec u_j$:
\begin{equation*}
\int_{\Om } D\vec u(\vec x)\cdot \vec\psi(\vec u(x)) 
=\int_{\Om_{\vec b}} \sgn(\det \nabla \vec u^{-1} (\vec y)) \adj \nabla \vec u^{-1}(\vec y)\cdot \vec\psi(\vec y) \, \dd \vec y.
\end{equation*} 
Hence,
\begin{equation*}
  \int_{\Om_{\vec b}}  \sgn(\det \nabla \vec u_j^{-1})\adj \nabla \vec u^{-1}_j (\vec y) \cdot \vec \psi ( \vec y) \, \dd\vec y \to  \int_{\Om_{\vec b}} \sgn(\det \nabla \vec u^{-1} (\vec y))\adj \nabla\vec u^{-1}(\vec y)\cdot \vec \psi ( \vec y) \, \dd\vec y .
\end{equation*} 
This proves the convergence of $ \sgn(\det \nabla \vec u_j^{-1})\adj \nabla \vec u^{-1}_j (\vec y)$ in the sense of measures,
and then in $L^1$ due to the equiintegrability. 
\end{proof}

The following lemma presents three matrix inequalities.
We first prove inequality \eqref{eq:sqrt3} and show that the constant is optimal.
However, we would need a constant $2$ instead of $\sqrt{3}$, since this would provide the right constant in front of the total variation in \eqref{eq:F}, as explained in the introduction and as shown by the construction in \cite{BaHEMoRo_24PUB}.

\begin{lemma}\label{le:better than 2}
\begin{enumerate}[label=\alph*)]
\item\label{item:sqrt3} $|\vec A|^2 \geq \sqrt{3} |\cof \vec A|$ for all $\vec A \in \R^{3 \times 3}$, with optimal constant.

\item\label{item:better2A} $|\vec A|^2 \geq 2 |\cof \vec A| - 2\max\{1, \det \vec A \}$ for all $\vec A \in \R^{3 \times 3}_+$.

\item\label{item:better2B}
$\frac{|\cof\vec B|^2}{\det \vec B} \geq 2 |\vec B| - 2\max\{1, \det \vec B\}$ for all $\vec B \in \R^{3 \times 3}_+$.
\end{enumerate}
\end{lemma}
\begin{proof}
For \ref{item:sqrt3} and \ref{item:better2A}, we note that the three terms $|\vec A|$, $|\cof \vec A|$ and $\det \vec A$ are invariant under multiplication by rotations.
Therefore, by singular value decomposition, we can assume that $\vec A$ is diagonal with positive diagonal elements $v_1\leq v_2 \leq v_3$.

The inequality of \ref{item:sqrt3} is equivalent to $|\vec A|^4 - 3 |\cof \vec A|^2 \geq 0$, and we have
\begin{align*}
 |\vec A|^4 - 3 |\cof \vec A|^2 & = \left( v_1^2 + v_2^2 + v_3^2 \right)^2 - 3 \left( v_2^2 v_3^2 + v_1^2 v_3^2 + v_1^2 v_2^2 \right) \\
 & = \frac{1}{2} \left[ (v_1^2 - v_2^2)^2 + (v_1^2 - v_3^2)^2 + (v_2^2 - v_3^2)^2 \right] \geq 0 .
\end{align*}
That the constant is optimal can be seen by considering $\vec A = \vec I$.

For \ref{item:better2A} we have
\begin{align*}
 |\vec A|^4 - 4 |\cof \vec A|^2 = \left( v_1^2 + v_2^2 + v_3^2 \right)^2 - 4 \left( v_2^2 v_3^2 + v_1^2 v_3^2 + v_1^2 v_2^2 \right) = \left (v_1^2 + v_2^2 - v_3^2\right)^2 -4 v_1^2 v_2^2 ,
\end{align*}
so
\[
 |\vec A|^4 \geq 4 |\cof \vec A|^2 - 4 v_1^2 v_2^2 .
\]
Taking into account that if $a\geq b\geq0$ then $\sqrt{a^2-b^2}\geq a-b$, from the inequality above we get
\begin{equation*}
|\vec A|^2 \geq 2 |\cof \vec A| - 2 v_1 v_2.
\end{equation*}
If $ v_1 v_2 \leq 1$ then we are done.
If $v_1 v_2 > 1$, then $v_3 \geq v_2 >1$ and therefore
\[
 1 < v_1 v_2 < v_1 v_2 v_3 = \det\vec A = \max\{1, \det \vec A \}.
\]
This shows \ref{item:better2A}. By taking
$\vec A=\vec B^{-1}=\frac{\adj \vec B}{\det \vec B}$, since $\cof\vec A=\frac{\vec B^T}{\det \vec B}$, we get from \ref{item:better2A} that 
\begin{equation*}
\frac{|\cof\vec B|^2}{(\det \vec B)^2} \geq 2 \frac{|\vec B|}{\det \vec B} - 2\max\{1, \frac{1}{\det \vec B}\},
\end{equation*}
and therefore \ref{item:better2B}.
\end{proof}

\medskip
We are now ready to prove the lower semicontinuity.
\begin{proposition}\label{prop:semicontinuity}
Let \(\{\vec u_j\}_j\) be a sequence in \(\B\) such that $\{ F (\vec u_j)\}_j$ is equibounded and \(\vec u_j \rightharpoonup \vec u\) in \(H^1(\Om,\R^3)\). Then
\begin{equation*}
\liminf_{j \to \infty} F(\vec u_j)\geq F(\vec u).
\end{equation*}
\end{proposition}
\begin{proof}
Consider the following functional defined for maps \(\vec w\) that are inverses of maps in \(\B\):
\begin{equation*}
\hat E(\vec w) := \int_{\Om_{\vec b}} \frac{|\cof \nabla \vec w|^2}{|\det \nabla \vec w|}\, \dd \vec y.
 \end{equation*}
By Corollary \ref{cor:area-formula} and Proposition \ref{prop:relation_inverses}, for \(\vec u \in \B\) we have $\int_\Om| D\vec u|^2 \, \dd \vec x=\hat E(\vec u^{-1})$.

By Corollary \ref{cor:Bcompact} the limit $\vec u$ belongs to $\B$. Let $V$ be the singular set of $D\vec u^{-1}$, i.e., $V \subset \Om_{\vec b}$ 
is any Borel set with $|V| = 0$ such that $ | D^s\vec u^{-1} | (\Om_{\vec b} \setminus V) = 0$.
As shown in Proposition \ref{pr:convergenza det e cof}, $\{ \det \nabla\vec u_j^{-1} \}_j$ is equiintegrable.
For $\e>0$, let $V_\e$ an open neighbourhood of $V$ such that $|V_\e|<\e$ and $\int_{V_\e}  |\det \nabla \vec u_j^{-1}|<\e$ for any $j\in \N$.

Next, by Lemma \ref{le:better than 2} \ref{item:better2B} we have
\begin{equation*}
\hat E(\vec u_j^{-1})
\geq \int_{\Om_{\vec b}\setminus V_\e} \frac{|\cof \nabla \vec u_j^{-1}|^2}{|\det \nabla \vec u_j^{-1}|} \, \dd \vec y
+2\int_{V_\e} | \nabla \vec u_j^{-1}|\, \dd \vec y- 4\e
 \end{equation*}
and, hence (recalling \eqref{eq:neo_Hook_energy} and \eqref{eq:F}),
\begin{align*}
 F (\vec u_j) & = 
 \hat E(\vec u_j^{-1}) + \int_{\Om} H (|\det D \vec u_j|) \, \dd \vec x + 2 \|D^s \vec u_j^{-1}\| \\
 & \geq \int_{\Om_{\vec b}\setminus V_\e} \frac{|\cof \nabla \vec u_j^{-1}|^2}{|\det \nabla \vec u_j^{-1}|}\, \dd \vec y + \int_{\Om} H (|\det D \vec u_j|) \, \dd \vec x 
+2 | D\vec u_j^{-1} |(V_\e)- 4\e .
\end{align*}
Let \(h : (0, \infty)^2 \to \R\) be defined as $h (a, b) = \frac{a^2}{b}$.
Since it is convex in $(a,b)$ and increasing in $a$, 
the map $\R^{3\times 3} \times (0,\infty) \ni (\vec M,b)\mapsto h(|\vec M|,b)$ is convex. 
Then, by writing
\begin{align*}
\int_{\Om_{\vec b}\setminus V_\e} \frac{|\cof \nabla \vec u_j^{-1}|^2}{|\det \nabla \vec u_j^{-1}|}\, \dd \vec y 
=\int_{\Om_{\vec b}\setminus V_\e} h(|\cof \nabla \vec u_j^{-1}|, |\det \nabla \vec u_j^{-1}|)\, \dd \vec y 
\end{align*}
and recalling that by Proposition \ref{pr:convergenza det e cof}
$\sgn(\det \nabla \vec u_j^{-1})\cof \nabla \vec u^{-1}_j 
\rightharpoonup \sgn(\det \nabla \vec u^{-1})\cof \nabla \vec u^{-1}$ in $L^1({\Om}_{\vec b},\R^{3\times 3})$ 
and $|\det \nabla \vec u^{-1}_j| \rightharpoonup  |\det \nabla \vec u^{-1}|$ in $L^1(\Om_{\vec b})$,
we find by the lower semicontinuity of convex functionals, see e.g.\  \cite[Theorem 5.14]{FoLe07}, that 
\begin{equation*}
\liminf_{j \to \infty}
\int_{\Om_{\vec b}\setminus V_\e} \frac{|\cof \nabla \vec u_j^{-1}|^2}{|\det \nabla \vec u_j^{-1}|}\, \dd \vec y \geq
\int_{\Om_{\vec b}\setminus V_\e} \frac{|\cof \nabla \vec u^{-1}|^2}{|\det \nabla \vec u^{-1}|}\, \dd \vec y.
\end{equation*}

Besides,  by  Proposition \ref{pr:limituj} we also have $\vec u_j^{-1} \weakcs \vec u^{-1}$ in $BV (\Om_{\vec b}, \R^3)$
and $|\det D \vec u_j| \rightharpoonup |\det D \vec u|$ in $L^1(\Om)$.
Hence, by using the convexity of $H$, we obtain
\[
 \liminf_{j\to\infty} F(\vec u_j) \geq \int_{\Om_{\vec b}\setminus V_\e} \frac{|\cof \nabla \vec u^{-1}|^2}{|\det \nabla \vec u^{-1}|}\, \dd \vec y + \int_{\Om} H (|\det D \vec u|) \, \dd \vec x
+2 | D\vec u^{-1} |(V_\e)- 4\e .
\]
By using that $| D\vec u^{-1} |(V_\e) \geq | D\vec u^{-1} |(V) = \| D^s \vec u^{-1} \|$, sending $\e$ to zero and going back to the reference configuration in the integral in $\Om_{\vec b}$, we conclude.
\end{proof}

With Corollary \ref{cor:Bcompact} and Proposition \ref{prop:semicontinuity} the proof of Theorem \ref{main theorem}
is immediate.

{\small
\bibliography{biblio} \bibliographystyle{siam}
}

\end{document}